\nonstopmode
\documentclass[10pt]{amsart}
\usepackage{latexsym}
\usepackage{amsmath, amssymb,amsthm}
\usepackage{color}
\usepackage[all]{xy}
\usepackage{pdflscape}
\usepackage{longtable}
\usepackage{rotating}
\usepackage{verbatim}
\usepackage{hyperref}
\usepackage{cleveref}
\usepackage{subfigure}
\usepackage{mathrsfs}
\usepackage{mdwlist}
\usepackage{dsfont}
\usepackage{bbold}
\usepackage{mathtools}
\usepackage{booktabs}
\usepackage{float}

\DeclareFontFamily{U}{mathb}{\hyphenchar\font45}
\DeclareFontShape{U}{mathb}{m}{n}{
      <5> <6> <7> <8> <9> <10> gen * mathb
      <10.95> mathb10 <12> <14.4> <17.28> <20.74> <24.88> mathb12
      }{}
\DeclareSymbolFont{mathb}{U}{mathb}{m}{n}
\DeclareFontSubstitution{U}{mathb}{m}{n}

\let\dot\relax
\DeclareMathAccent{\dot}{0}{mathb}{"39}
\let\ddot\relax
\DeclareMathAccent{\ddot}{0}{mathb}{"3A}
\let\dddot\relax
\DeclareMathAccent{\dddot}{0}{mathb}{"3B}
\let\ddddot\relax
\DeclareMathAccent{\ddddot}{0}{mathb}{"3C}

\theoremstyle{plain}
\newtheorem*{theorem*}{Theorem}
\newtheorem{theorem}{Theorem}[section]
\newtheorem*{lemma*}{Lemma}
\newtheorem{lemma}[theorem]{Lemma}
\newtheorem*{proposition*}{Proposition}
\newtheorem{proposition}[theorem]{Proposition}
\newtheorem*{corollary*}{Corollary}

\newtheorem*{claim*}{Claim}

\newtheorem*{conjecture*}{Conjecture}

\newtheorem*{question*}{Question}
\theoremstyle{definition}
\newtheorem*{definition*}{Definition}
\newtheorem{definition}[theorem]{Definition}
\newtheorem*{example*}{Example}

\newtheorem*{algorithm*}{Algorithm}
\newtheorem*{remark*}{Remark}
\newtheorem*{remarks*}{Remarks}
\newtheorem{remark}[theorem]{Remark}

\newtheorem*{convention*}{Convention}

\theoremstyle{plain} 
\newcommand{\thistheoremname}{}
\newtheorem{genericthm}[theorem]{\thistheoremname}

\newtheorem*{genericthm*}{\thistheoremname}
\newenvironment{namedthm*}[1]
{\renewcommand{\thistheoremname}{#1}%
	\begin{genericthm*}}
	{\end{genericthm*}}

\numberwithin{equation}{section}

\def\al{\alpha}
\def\be{\beta}
\def\ga{\gamma}
\def\de{\delta}

\def\et{\eta}
\def\th{\theta}

\def\rh{\rho}

\def\ta{\tau}

\def\ph{\phi}

\def\ps{\psi}
\def\om{\omega}

\def\Ga{\Gamma}
\def\De{\Delta}

\def\Ph{\Phi}
\def\Ps{\Psi}

\def\N{\mathbb{N}}

\def\R{\mathbb{R}}
\def\Z{\mathbb{Z}}

\def\cB{\mathcal{B}}
\def\cC{\mathcal{C}}

\def\cF{\mathcal{F}}
\def\cG{\mathcal{G}}
\def\cH{\mathcal{H}}

\def\cZ{\mathcal{Z}}

\def\fX{\mathfrak{X}}

\def\p{\partial}

\def\<{\langle}
\def\>{\rangle}

\def\Id{\on{Id}}

\let\on=\operatorname

\newcommand{\sr}[1]%
{\ifmmode{}^\dagger\else${}^\dagger$\fi\ifvmode
\vbox to 0pt{\vss
 \hbox to 0pt{\hskip\hsize\hskip1em
 \vbox{\hsize3cm\raggedright\pretolerance10000
 \noindent #1\hfill}\hss}\vss}\else
 \vadjust{\vbox to0pt{\vss%
 \hbox to 0pt{\hskip\hsize\hskip1em%
 \vbox{\hsize3cm\raggedright\pretolerance10000%
 \noindent #1\hfill}\hss}\vss}}\fi%
}

\providecommand{\mapsfrom}{\kern.2em%
\setbox0=\hbox{$\leftarrow$\kern-.10em\rule[0.26mm]{0.1mm}{1.3mm}}\box0%
\kern.3em}

\title[Besov vector fields and their flows]{On time-dependent Besov vector fields and the regularity of their flows}

\author{David Nicolas Nenning}

\address{D.N.~Nenning: Fakult\"at f\"ur Mathematik, Universit\"at Wien, 
Oskar-Morgenstern-Platz~1, A-1090 Wien, Austria}
\email{david.nicolas.nenning@univie.ac.at}

\begin{document}
	\thanks{The author was supported by FWF-Project P~26735-N25}
	\date{\today}
	
	\maketitle
	
	\begin{abstract}
		We show ODE-closedness for a large class of Besov spaces $\cB^{n,\al,p}(\R^d,\R^d)$, where $n \ge 1,~\al \in (0,1],~p \in [1,\infty]$. ODE-closedness means that \emph{pointwise time-dependent $\cB^{n,\al,p}$-vector fields} $u$ have unique flows $\Ph_u \in \Id+\cB^{n,\al,p}(\R^d,\R^d)$. The class of vector fields under consideration contains as a special case the class of Bochner integrable vector fields $L^1(I, \cB^{n,\al,p}(\R^d,\R^d))$.
		In addition, for $n \ge 2$ and $\al < \be$, we show continuity of the flow mapping $L^1(I,\cB^{n,\be,p}(\R^d,\R^d)) \rightarrow C(I,\cB^{n,\al,p}(\R^d,\R^d)), ~ u \to \Ph_u-\Id$. We even get $\ga$-H\"older continuity for any $\ga < \be - \al$.
	\end{abstract}

	\section{Introduction}
	One frequently used tool in computational anatomy to construct large deformation diffeomorphisms of $\R^d$ consists of solving an ODE
	\begin{align}
	\label{ode}
	\begin{split}
	\p_t \Ph(t) &= u(t,\Ph(t)),\\
	\Ph (0) &= x
	\end{split}
	\end{align}
	for a vector field $u:I \times \R^d \rightarrow \R^d$, which for some fixed time $t_0$, as a function of the initial condition $x$, gives rise to a diffeomorphism. We write $\Ph_u(t_0,x)=x
	+\ph_u(t_0,x)$ for the solution of \eqref{ode} at time $t_0$, and call $\Ph_u(t_0, \cdot)$ the \emph{flow of $u$ at time $t_0$}. For an introduction to this matter, we refer to \cite{Younes10}. \par
	Now assume $u$ has some prescribed regularity with respect to (time and) the spatial variable, then one naturally asks the  
	\emph{(motivating) question}:\par
	\vspace{1mm}
	\begin{center}
		\emph{What can be said about the regularity of $\Ph_u(t_0,\cdot)$?}
	\end{center}
	\vspace{1mm}
	If $\Ph_u$ is of the same regularity as $u$, this leads to the concept of \emph{ODE-closedness}.\par 
	In more detail: Let $E$ be a vector space of functions from $\R^d$ to $\R^d$ and $\cF$ a family of time-dependent vector fields $u:I \times \R^d \rightarrow \R^d$ such that $u(t,\cdot) \in E$ for all $t \in I$. If for all $u \in \cF$ and all $t \in I$, $\ph_u(t,\cdot)$ is again in $E$, we call the space $E$ \emph{ODE-closed}. Of course this notion depends on $\cF$ and becomes stronger if one enlarges the class $\cF$. ODE-closedness has already been studied for many spaces, e.g.
	\begin{itemize}
		\item $C^n_0$, in \cite{Younes10},
		\item H\"older spaces, in \cite{NenningRainer16},
		\item Sobolev spaces, in \cite{BruverisVialard14},
		\item several classes of (un)weighted smooth function spaces, in \cite{NenningRainer17}, \dots
	\end{itemize}
	Similar questions have also been considered for several classes of (infinite dimensional) Lie groups in \cite{Glockner16}.\par 
	The main goal of this paper is to show that Besov spaces $B^s_{\infty,p}(\R^d,\R^d)$ with $s > 1$ and $~p \in [1,\infty]$, or in our notation $\cB^{n,\al,p}(\R^d,\R^d)$, where $n = \lfloor s \rfloor, \al = s-\lfloor s \rfloor$, are ODE-closed (with respect to the class of pointwise time-dependent Besov vector fields). On the one hand (for $\al \in (0,1)$, $p < \infty$) this generalizes results from \cite{NenningRainer16}, on the other hand it contains as a special case (for $\al=1, ~p = \infty$) ODE-closedness of \emph{Zygmund spaces}.\par 
	We write $\fX_{n,\al,p}$ for the family of pointwise time-dependent Besov vector fields, which consists of all functions $u:I\times \R^d \rightarrow \R^d$ such that
	\begin{itemize}
		\item $u(t, \cdot) \in \cB^{n,\al,p}(\R^d,\R^d)$ for every $t \in I$,
		\item $u(\cdot,x)$ is measurable for every $x\in \R^d$,
		\item $I \ni t \to \|u(t,\cdot)\|_{\cB^{n,\al,p}(\R^d,\R^d)}$ is (Lebesgue) integrable. 
	\end{itemize}
	That this is a rather large class is illustrated by the fact that it contains $L^1(I,\cB^{n,\al,p}(\R^d,\R^d))$, the space of Bochner integrable vector fields into $\cB^{n,\al,p}(\R^d,\R^d)$, as a proper subset. Since we only require integrability with respect to time, we have to weaken the concept of solution of \eqref{ode}, i.e. we say $\Ph$ solves \eqref{ode} if it is an absolutely continuous function on $I$ and satisfies
	\begin{equation}
	\label{intode}
		\Ph(t) = x + \int_0^t u(s,\Ph(s))\,ds
	\end{equation}
	for all $t \in I$. \par 
	We give a precise answer to the motivating question for Besov spaces in our main theorem.
	\begin{namedthm*}{Theorem 1}
		\label{zygodeclosed}
		For all $n\in \N_{\ge 1},~\al \in (0,1],~ p \in [1,\infty]$, the Besov space $\cB^{n,\al,p}(\R^d,\R^d)$ is ODE-closed. For a fixed $u \in \fX_{n,\al,p}$, the mapping $t \mapsto \ph_u(t, \cdot)$ belongs to $C(I,\cB^{n,\al,p}(\R^d,\R^d))$.	
	\end{namedthm*}

	The proof of Theorem 1 heavily depends upon embedding and composition properties of Besov spaces. Using those, we are able to derive from \eqref{intode}, with the help of known results for H\"older spaces from \cite{NenningRainer16}, an integral inequality for the Besov norm of $\Ph_u$. To this inequality, we apply Gronwall's inequality and get uniform bounds. By a similar reasoning we are able to show continuity with respect to time. \par 
	
	Restricting to Bochner integrable vector fields, we are also able to show continuity results with respect to the vector field. This is the content of our second result.
	
	\begin{namedthm*}{Theorem 2}
		Let $0<\al < \be \le 1$ and $n \in \N_{\ge 2}$. Then 
		\[
		\on{Fl}: L^1(I,\cB^{n,\be,p}) \rightarrow C(I,\cB^{n,\al,p}), ~ u \mapsto \ph_u
		\]
		is continuous, even H\"older continuous of any order $\ga < \be-\al$.
	\end{namedthm*}

	\begin{remark}[A possible application]
		Theorems 1 and 2 may aid as a tool to study the \emph{geodesic problem} on (the group of) Besov diffeomorphisms
		\[
		\on{Diff}\cB^{n,\al,p}:= \{\Ph \in \Id+\cB^{n,\al,p}: \inf_{x \in \R^d} \det d_x\Ph(x)>0\},
		\]
		which is a Banach manifold modelled on the open set $\on{Diff}\cB^{n,\al,p} - \Id \subseteq \cB^{n,\al,p}$.
		Due to Theorem 1, the so-called \emph{Trouv\'e group} $\cG_{n,\al,p}:=\{\Ph_u(1,\cdot):u\in \fX_{n,\al,p}\}$, cf. \cite{NenningRainer17} for more details, satisfies
		$
		\cG_{n,\al,p} \subseteq (\on{Diff}\cB^{n,\al,p})_0,
		$
		where $(\cdot)_0$ denotes the connected component of the identity. It was shown in \cite{NenningRainer16} for $\al \in (0,1)$ that $\on{Diff}\cB^{n,\al,\infty}$ is a group with smooth right translations but only locally bounded left translations and inversion (observe that Lemma \ref{normequiv} gives that $\cB^{n,\al,\infty}$ coincides with the H\"older space of order $(n,\al)$). But this was enough to show $(\on{Diff}\cB^{n,\al,\infty})_0 \subseteq \cG_{n,\al,\infty}$. Adapting the proofs for $p < \infty $ and $\al = 1$ shouldn't cause too many problems. So let us assume for now we already know these regularity statements for $p < \infty$.\par 
		Then putting an inner product (or more generally an asymmetric norm) $g$ on $\cB^{n,\al,p}$, one may extend it to a right invariant Riemannian metric (or Finsler metric) $G$ on the tangent bundle $T \on{Diff}\cB^{n,\al,p}= \on{Diff}\cB^{n,\al,p} \times \cB^{n,\al,p}$ via $G(\Ph,f) = G(\Id, f \circ \Ph^{-1}):= g( f \circ \Ph^{-1})$. Observe that in general this only yields a locally bounded metric. Now the geodesic problem consists of finding curves of minimal length connecting two given diffeomorphisms $\Ph, ~\Ps \in \on{Diff}\cB^{n,\al,p}$, i.e.: Define $\on{dist}(\Ph,\Ps):= \inf\{l(\ga): \ga \in \Ga_{\Ph,\Ps} \}$, where $\Ga_{\Ph,\Ps}$ is the set of piecewise smooth curves $\ga:[0,1]\rightarrow \on{Diff}\cB^{n,\al,p}$ with $\ga(0)=\Ph$ and $\ga(1)= \Ps$, and
		\[
		l(\ga):=\int_0^1 G(\ga(t),\dot{\ga}(t)) \, dt.
		\]
		And right invariance gives $l(\ga)= \int_0^1 G(\Id, \dot{\ga}(t) \circ \ga(t)^{-1})\,dt$, which yields due to Theorem 1
		\[
		\on{dist}(\Ph,\Ps)= \inf\bigg\{ \int_0^1 G(\Id, u(t))\,dt: u \in \fX_{n,\al,p}, \Ps = \Ph_u(1,\cdot) \circ \Ph \bigg\}.
		\]
		Thus one can translate the problem of finding curves of minimal length to finding minimal vector fields (with respect to $G$). But there are several obstructions to overcome, e.g.: It is not clear whether the translation process is continuous (cf. Theorem 2). In addition it is also not clear which right-invariant metrics are smooth (the construction a priori only gives locally bounded metrics). 	
	\end{remark}

	The paper is structured as follows. In Section 2 we fix notation and introduce the spaces under consideration. In Section 3 we discuss continuous inclusion of Besov spaces and give elementary proofs. In Section 4 we formulate the necessary regularity results concerning composition. Finally, in Section 5 and 6, we  prove our main results, namely Theorems 1 and 2.

	\subsection*{Acknowledgments} I would like to express my gratitude to Armin Rainer for many helpful discussions throughout the preparation of this article.
		
	\section{Preliminaries}
	\subsection{Notation}
	In what follows $\N=\{0,1,2,\dots\}$ and $\N_{\ge 1}= \N \backslash \{0\}$. For two Banach spaces $E,F$, the space of $n$-times continuously Fr\'echet differentiable mappings from $E$ to $F$ shall be denoted by $C^n(E,F)$. For $f \in C^n(E,F)$, we write for the $k$-th derivative $f^{(k)}=D^k f = d^kf=d^k_x f$ which then is an element of $C^{n-k}(E, L_k(E;F))$. For $E_1, \dots, E_k$ Banach spaces, $L_k(E_1,\dots,E_k;F)$ is the space of bounded $k$-linear mappings from $E_1\times \cdots \times E_k \rightarrow F$ endowed with the norm
	\[
	\|T\|_{L_k(E_1,\dots, E_k;F)}:= \sup_{h_i \in E_i, h_i \neq 0}\frac{\|T(h_1,\dots,h_k)\|}{\|h_1\|\cdots\|h_k\|}.
	\]
	If $E_i=E$ for all $i$, we simply write $L_k(E;F)$. As usual $\|\cdot\|_\infty$ shall denote the sup-norm of a function. We set
	\[
	C^n_b(E,F):= \{f \in C^n(E,F): \|f\|_{C^n_b}:=\sup_{0\le k \le n}\|f^{(k)}\|_\infty  < \infty  \}.
	\]
	\subsection{Besov spaces}
	Besov spaces are one way of filling the gap between $C^{n+1}_b(E,F)$ and $C^{n}_b(E,F)$. But before we can define them, we need some additional notation.\par 
	For a function $f:E\rightarrow F$, we set $\De^k_{h}f(x)$ to be the $k$-th difference of $f$ at $x$ by $h$. They are defined inductively by
	\[
	\De^1_{h}f(x):= f(x+h)-f(x), ~\De^{k+1}_{h}f(x):= \De^1_h \De^k_hf(x),
	\]
	especially $\De^2_h f(x)= f(x+2h)-2f(x+h)+f(x)$. Having this, we set
	\[
	\om(f;t):=\sup_{ x \in E, \|h\|\le t} \|\De^1_h f(x) \|
	\]
	and
	\[
	\et(f;t):=\sup_{ x \in E, \|h\|\le t} \|\De^2_h f(x)\|.
	\]
	For $p \in [1,\infty)$, $\al \in (0,1]$ 
	\[
	Z^{\al,p}(f):= \bigg(\int_0^\infty \bigg(\frac{\et(f;t)}{t^\al}\bigg)^p \frac{dt}{t}\bigg)^{1/p},
	\]
	and
	\[
	Z^{\al,\infty}(f):= \sup_{t>0} \frac{\et(f;t)}{t^\al}.
	\]
	In addition $H^{\al,p}(f)$ is defined by replacing $\et$ with $\om$. We set
	\[
	\cB^{n,\al,p}(E,F):=\bigg\{f \in C^n(E,F): \|f\|^\cZ_{n,\al,p}:= \sup\{ \|f\|_{C^n_b(E,F)}, Z^{\al,p}(f^{(n)}) \}<\infty \bigg\},
	\]
	the \emph{$p$-Besov space of order $(n,\al)$}. For $p = \infty$, we also write $\cZ^{n,\al}(E,F)$ and call it the \emph{Zygmund space of order $(n,\al)$}. Furthermore we introduce
	\[
	\cB^{n,\al,p}_\cH(E,F) := \bigg\{f \in C^n(E,F): \|f\|^\cH_{n,\al,p}:= \sup\{ \|f\|_{C^n_b(E,F)}, H^{\al,p}(f^{(n)}) \}<\infty \bigg\}.
	\]
	For $p = \infty$, we also write $\cC^{n,\al}(E,F)$ and call it the \emph{H\"older space of order $(n,\al)$}. Actually the spaces $\cB^{n,\al,p}(E,F)$ and $\cB^{n,\al,p}_\cH(E,F)$ are not equal only for $\al = 1$. If domain and codomain are clear from the context, we sometimes omit mentioning them explicitly, e.g. $\cB^{n,\al,p}$ instead of $\cB^{n,\al,p}(E,F)$. 
	
	\begin{remark}
		For $E=\R^d,~F = \R$, and $q<\infty$, one can also consider $\et_q(f;t):= \sup_{\|h\|\le t} \|\De^2_h f\|_{L^q}$, i.e. $\et_\infty$ coincides with $\et$ as defined above, and define $Z^{\al,p}_q(f)$ as integral over $\et_q(f;t)$. Then $\cB^{n,\al,p}_q(\R^d,\R)$ shall denote the space of $n$-times weakly differentiable functions whose derivatives are in $L^q$ and such that $Z^{\al,p}_q(f^{(n)})$ is finite. This then even gives a more (or the most) general notion of Besov space. Since we work with pointwise estimates, we restrict ourselves to the case $q = \infty $.
	\end{remark}

\subsection{Fa\`a di Bruno's formula}
We shall make frequent use of Fa\`a di Bruno's formula, whose Banach space version reads as follows.

\begin{proposition}
	Let $E,F,G$ be Banach spaces, let $f: E \supseteq U  \rightarrow F$ and $g: F \supseteq V \rightarrow G$ 
	be $k$ times Fr\'echet differentiable, and assume $f(U)\subseteq V$. Then $g\circ f:U \rightarrow G$ is $k$ 
	times Fr\'echet differentiable, and for all $x \in U$,		
	\begin{align}
	\label{eq:FaadiBruno}
	d^k(g\circ f)(x) = \on{sym}
	\sum_{l=1}^{k} \sum_{ \ga \in \Ga(l,k)} c_{\ga} g^{(l)}(f(x))\left(f^{(\ga_1)}(x), \dots, f^{(\ga_l)}(x)\right),
	\end{align}
	where $\Ga(l,k) := \{\ga \in \N_{>0}^l : |\ga| = k\}$,
	$c_{\ga} := \frac{k!}{l! \ga!}$, and $\on{sym}$ denotes symmetrization of multilinear mappings.
\end{proposition}

\section{Inclusion of Besov spaces} In this section we summarize inclusion relations of Besov spaces relevant to our forthcoming analysis. The results are essentially all known, but we nevertheless include proofs. First of all for the readers convenience and secondly since we only want to use the definition of Besov spaces given above. In the literature most of the proofs are based on some different, but equivalent, definition.

\subsection{Relation of $\cB^{n,\al,p}$ and $\cB_\cH^{n,\al,p}$}

	\begin{lemma}
		\label{normequiv}
		For $\al \in (0,1)$ and $f \in C_b(E,F)$
		\begin{equation}
		\label{eq:normequiv}
		\frac{1}{2} Z^{\al,p}(f) \le H^{\al,p}(f) \le \frac{1}{1-2^{\al-1}} Z^{\al,p}(f),
		\end{equation}
		where the l.h.s. inequality also holds for $\al = 1$. This implies 
		for $\al \in (0,1)$
		\[
		\cB^{n,\al,p}_\cH \cong \cB^{n,\al,p},
		\]
		and
		\[
		\cB^{n,1,p}_\cH \hookrightarrow \cB^{n,1,p}.
		\]
		But in general $\cB^{n,1,p}_\cH \subsetneq \cB^{n,1,p}$.
	\end{lemma}
	
	\begin{proof}
		$Z^{\al,p}(f) \le 2H^{\al,p}(f)$ for $\al \in (0,1]$ follows easily by applying the triangle inequality. This already shows $\cB^{n,\al,p}_\cH \hookrightarrow \cB^{n,\al,p}$.\par 
		We are left to show the second inequality in \eqref{eq:normequiv}. So now let $\al \in (0,1)$ and assume $Z^{\al,p}(f)<\infty$ (otherwise there is nothing to prove) for $p < \infty$. 

		We need the following observation: For every $n \in \N_{\ge 1}$ and $0 \le k \le n-1$
		\begin{align*}
		\begin{split}
			\bigg(\int_0^\infty& \bigg(\frac{\sup_{x\in E, \|h\|\le \ta} \|2^k(\De^2_{2^{n-k-1}h}f(x))\|}{\ta^\al}\bigg)^p\frac{d\ta}{\ta}\bigg)^{1/p} \le 2^{k+(n-k-1)\al}Z^{\al,p}(f),
		\end{split}
		\end{align*}
		which follows immediately via the change of variables $s = 2^{n-k-1}\ta$. This implies 
		\begin{align}
		\begin{split}
		\label{eq:aux1}
		\bigg(\int_0^\infty& \bigg(\frac{\sup_{x\in E, \|h\|\le \ta} \|f(x+2^n h)-f(x)-2^n(f(x+h)-f(x))\|}{\ta^\al}\bigg)^p\frac{d\ta}{\ta}\bigg)^{1/p}\\
		&=\bigg(\int_0^\infty \bigg(\frac{\sup_{x\in E, \|h\|\le \ta} \| \sum_{k = 0}^{n-1} 2^k\De^2_{2^{n-k-1}h} f(x) \|}{\ta^\al}\bigg)^p\frac{d\ta}{\ta}\bigg)^{1/p}\\
		&\le \bigg(\sum_{k=0}^{n-1}2^{k+(n-k-1)\al}\bigg)Z^{\al,p}(f) \le 2^{n-1} \frac{1}{1-2^{\al-1}} Z^{\al,p}(f).
		\end{split}
		\end{align}
		
		Now let $\de>0$ be fixed and choose $n \in \N$ such that $\frac{1}{2^n} \le \frac{(\al p)^{1/p} \de^{1+\al}}{2\|f\|_\infty}$, then 
		\begin{align*}
			&\bigg(\int_{\de}^{\infty} \bigg(\frac{\om(f;\ta)}{\ta^\al}\bigg)^p\frac{d\ta}{\ta}\bigg)^{1/p}\\
			&\le \frac{1}{2^n} \bigg(\int_0^\infty \bigg(\frac{\sup_{x\in E, \|h\|\le \ta} \|f(x+2^n h)-f(x)-2^n(f(x+h)-f(x))\|}{\ta^\al}\bigg)^p\frac{d\ta}{\ta}\bigg)^{1/p}\\
			&\qquad +  \frac{1}{2^n} \bigg( \int_{\de}^{\infty} \bigg(\frac{\sup_{x \in E, \|h\|\le \ta}\|f(x+2^n h)-f(x)\|}{\ta^\al}\bigg)^p\frac{d\ta}{\ta}\bigg)^{1/p}\\
			&\le \frac{1}{1-2^{\al-1}} Z^{\al,p}(f) + \frac{1}{2^n}2\|f\|_\infty \bigg(\int_{\de}^{\infty} \frac{1}{\ta^{\al p + 1}}\,d\ta\bigg)^{1/p}\\
			&= \frac{1}{1-2^{\al-1}} Z^{\al,p}(f) + \frac{1}{2^n} \frac{2\|f\|_\infty}{(\al p)^{1/p} \de^\al} \le \frac{1}{1-2^{\al-1}} Z^{\al,p}(f) + \de,
		\end{align*}
		and as $\de$ tends to $0$ in the above inequality, we get the desired result. The proof for $p = \infty$ works analogously.\par 
		
		Finally the famous Weierstra{\ss} function 
		$
		f_b(x):= \sum_{k=0}^\infty b^{-k}\cos( b^k \pi x),
		$
		originally only considered for certain integer values of $b$ and by Hardy for arbitrary $b>1$ (cf. \cite{Hardy16}),
		is an element of $\cB^{0,1,\infty}(\R,\R)$ but nowhere differentiable and thus by Rademacher's theorem not Lipschitz and therefore not in $\cB^{0,1,\infty}_\cH$. This shows $\cB^{0,1,\infty}_\cH(\R,\R) \subsetneq \cB^{0,1,\infty}(\R,\R)$.\par 
		By multiplying $f_b$ with a smooth cutoff function and integrating, we get a function in $\cB^{1,1,\infty}(\R,\R) \backslash \cB^{1,1,\infty}_\cH(\R,\R)$. And iterating this procedure yields strict inclusions for parameters $(n,1,\infty)$. In order to get $\cB^{n,1,p}_\cH(\R,\R) \subsetneq \cB^{n,1,p}(\R,\R)$ for $p < \infty$, we may use the function from Remark \ref{rem:ex}:\par 
		There we construct $f \in \cB^{1,\al,\infty}_\cH(\R,\R) \backslash \bigcup_{1\le p < \infty}\cB^{0,1,p}_\cH(\R,\R)$. And since $\cB^{1,\al,\infty}_\cH(\R,\R) \hookrightarrow \cB^{1,\al,\infty}(\R,\R)  \hookrightarrow \cB^{0,1,p}(\R,\R)$ for any $1 \le p \le \infty$, where the second inclusion is due to Proposition \ref{embedding}, we immediately get $\cB^{0,1,p}_\cH(\R,\R)\subsetneq \cB^{0,1,p}(\R,\R)$. The construction for $n \ge 1$ is analogous to the case $p = \infty$ above.
	\end{proof}
	
	\begin{remark}
		By definition a continuous increasing function $\rh:[0,\infty] \rightarrow [0,\infty]$ with $\rh(0)=0$ is a \emph{modulus of continuity} of a function $f:E \rightarrow F$ iff  $\om(f;t) \le \rh(t)$ for all $t\ge 0$. Now it is clear by definition that any $f \in \cC^{0,1}(E,F)$ admits a modulus of continuity $\rh(t)=Ct$ for some $C>0$, whereas a function $g \in \cZ^{0,1}(\R,\R)$ in general only admits a modulus of continuity $\rh(t)=Dt |\log(t)|$ for some $D>0$, cf. \cite[Theorem 10]{Zygmund45}. This is of course less restrictive.
	\end{remark}
	
	\subsection{Relation of $\cB^{n,\be,q}$ and $\cB^{m,\al,p}$}
	Our forthcoming analysis will heavily depend upon the following embedding theorem, which corresponds to \cite[ch. 3, Theorem 4]{Peetre76}. 
	
	\begin{proposition}
		\label{embedding}
		Let $E,F$ be Banach spaces, and $m,n \in \N$ and $\al,\be \in (0,1]$ be such that $m+\al < n+\be$. Then
		\begin{equation}
		\label{regdec}
		\cB^{n,\be,q}(E,F) \hookrightarrow \cB^{m,\al,p}(E,F)
		\end{equation}
		for any $p,q \in [1,\infty]$. In addition, for $1 \le q \le p \le \infty$
		\begin{equation}
		\label{intinc}
		\cB^{n,\al,q}(E,F) \hookrightarrow \cB^{n,\al,p}(E,F).
		\end{equation}
		There even is a constant $C$ independent of $n,\al,p,q$ such that
		\[
		Z^{\al,p}(f^{(n)}) \le C Z^{\al,q}(f^{(n)}),
		\]
		i.e. $\|f\|_{n,\al,p}\le C \|f\|_{n,\al,q}$ independent of all parameters. The constant for the embedding \eqref{regdec} depends on $n,m,\al,\be$.
	\end{proposition}
	
	\begin{remark}
		\label{hölderrem}
		For $m = n$, \eqref{regdec} is also valid for $\cB_\cH$, i.e. $\cB_\cH^{n,\be,q}(E,F) \hookrightarrow \cB^{n,\al,p}_\cH(E,F)$. \eqref{intinc} holds for $\cB_\cH$ without any restrictions.
	\end{remark}
	
	\begin{proof}[Proof of Proposition \ref{embedding}]
		Let us first prove \eqref{intinc}. To this end we show the following intermediate result, inspired by Exercise 14.21 from \cite{Leoni09}: Given an increasing function $g:[0,\infty) \rightarrow [0,\infty)$, $1\le \th \le 2$, and $1\le q < p < \infty$, there exists a constant $C$ independent from $g,p,q,\th$ such that
		\begin{equation}
		\label{pqest}
		\bigg(\int_0^\infty \bigg(\frac{g(\ta)}{\ta^\th}\bigg)^p\frac{d\ta}{\ta}\bigg)^{1/p} \le C \bigg(\int_0^\infty \bigg(\frac{g(\ta)}{\ta^\th}\bigg)^q\frac{d\ta}{\ta}\bigg)^{1/q},
		\end{equation}
		and also for $p = \infty$, when the left hand side in \eqref{pqest} is interpreted as $\sup_{0 < \ta } \frac{g(\ta)}{\ta^\th}$. Let us show \eqref{pqest} for $p < \infty$ first. To this end, first note that $l^q \hookrightarrow l^p$, and it even holds $\|a\|_{l^p}\le \|a\|_{l^q}$ for any $a \in \R^\Z$. We can estimate
		\begin{align*}
			\begin{split}
				\frac{1}{2^{1/p}} \|(g(2^k)2^{-\th(k+1)})_k\|_{l^p} &=\frac{1}{2^{1/p}} \bigg(\sum_{k=-\infty}^{\infty} (g(2^k) 2^{-\th(k+1)})^p\bigg)^{1/p}\\
				& \le \bigg(\int_0^\infty \bigg(\frac{g(\ta)}{\ta^\th}\bigg)^p\frac{d\ta}{\ta}\bigg)^{1/p}\\
				&\le \bigg(\sum_{k=-\infty}^{\infty} (g(2^{k+1}) 2^{-\th k})^p\bigg)^{1/p}\\
				&= \|(g(2^{k+1})2^{-\th k})_k\|_{l^p}\\
				&= 2^{2\th}\|(g(2^k)2^{-\th(k+1)})_k\|_{l^p}
			\end{split}
		\end{align*}
		
		where we just used that $g$ is increasing. Clearly the same inequality holds for $p$ replaced by $q$. Thus we get
		\begin{align*}
			\bigg(\int_0^\infty \bigg(\frac{g(\ta)}{\ta^\th}\bigg)^p\frac{d\ta}{\ta}\bigg)^{1/p} &\le 2^{2\th}\|(g(2^k)2^{-\th(k+1)})_k\|_{l^p}\\
			&\le 2^{2\th}\|(g(2^k)2^{-\th(k+1)})_k\|_{l^q}\\
			&\le \underbrace{2^{2\th + 1/p}}_{\le 2^{2\th+1}} \bigg(\int_0^\infty \bigg(\frac{g(\ta)}{\ta^\th}\bigg)^q\frac{d\ta}{\ta}\bigg)^{1/q}.
		\end{align*}
		Finally, assume $p = \infty$, and let $\ta_0 \in (0,\infty)$. Then there exist unique $k_0 \in \Z$ and $t_0 \in [1,2^{1/k_0})$ such that $\ta_0 = (t_0 2)^{k_0}$ and we can estimate
		\begin{align*}
			\frac{g(\ta_0)}{\ta_0^\th} &= \frac{g((t_0 2)^{k_0})}{(t_0 2)^{\th k_0}} \le \bigg(\sum_{k=-\infty}^{\infty} (g((t_0 2)^k) (t_0 2)^{-\th k})^q\bigg)^{1/q}\\
			&= (t_0 2)^\th \bigg(\sum_{k=-\infty}^{\infty} (g((t_0 2)^k) (t_0 2)^{-\th (k+1)})^q\bigg)^{1/q}\\
			&\le 2^{2\th+1} \bigg(\int_0^\infty \bigg(\frac{g(\ta)}{\ta^\th}\bigg)^q\frac{d\ta}{\ta}\bigg)^{1/q}.
		\end{align*}
		Since the constant is independent from $\ta_0$, this finishes the proof of \eqref{pqest} also for $p = \infty$, and shows that $2^5$ is a suitable choice for $C$.\par 
		Inequality \eqref{pqest} now immediately gives a proof for \eqref{intinc}: Let $f \in \cB^{n,\al,q}(E,F)$, and set $g(\ta):= \et(f^{(n)};\ta)\ta$, which is increasing, and $\th := \al + 1$. Then \eqref{pqest} yields $Z^{\al,p}(f^{(n)}) \le C Z^{\al,q}(f^{(n)})$, which gives \eqref{intinc}. By choosing $g(\ta):= \om(f^{(n)};\ta)\ta$, we also get a proof for the respective statement in Remark \ref{hölderrem}.\par 

		Now to the proof of \eqref{regdec}. Assume first that $n = m$ (which implies that $\al < \be$). Due to \eqref{intinc} it is enough to show that $\cB^{n,\be,\infty}(E,F) \hookrightarrow \cB^{n,\al,1}(E,F)$. So let $f \in \cB^{n,\be,\infty}(E,F)$. Then we estimate
		\begin{align*}
			\int_0^1 \frac{\et(f^{(n)};\ta)}{\ta^\al} \frac{d\ta}{\ta} &= \int_0^1 \frac{\et(f^{(n)};\ta)}{\ta^\be} \frac{d\ta}{\ta^{1-(\be-\al)}}\\
			&\le \sup_{\ta > 0} \frac{\et(f^{(n)};\ta)}{\ta^\be} \int_0^1\frac{1}{\ta^{1-(\be-\al)}} \, d\ta\\
			&= \frac{1}{\be-\al} Z^{\be,\infty}(f^{(n)}),
		\end{align*}
		and since $\int_1^\infty \frac{\et(f^{(n)};\ta)}{\ta^\al} \frac{d\ta}{\ta} \le \frac{4\|f^{(n)}\|_\infty}{\al}$, we can easily conclude that 
		\[
		\|f\|_{n,\al,1}^\cZ \le 4\frac{\be}{\al(\be-\al)} \|f\|_{n,\be,\infty}^\cZ,
		\]
		and by \eqref{intinc} the existence of an absolute constant $B$ such that
		\[
		\|f\|_{n,\al,p}^\cZ \le B \frac{\be}{\al(\be-\al)} \|f\|_{n,\be,q}^\cZ.
		\]
		Again by replacing $\et$ with $\om$, we get a proof for the respective statement in Remark \ref{hölderrem}.\par 
		Next we show $\cB^{n+1,\be,\infty} \hookrightarrow \cB^{n,1,1}$ for any $\be \in (0,1)$, which together with the preceding proof for $n = m$ and \eqref{intinc} then yields  \eqref{regdec}. Let $f \in \cB^{n+1,\be,\infty}$. Again we estimate (using the mean value theorem)
		\begin{align*}
			\int_0^1& \frac{\et(f^{(n)};\ta)}{\ta} \frac{d\ta}{\ta}\\
			 & \le \int_0^1 \frac{2\sup_{x\in E, |h|\le \ta} \|f^{(n)}(x+h)-f^{(n)}(x) - f^{(n+1)}(x)(h)\|}{\ta} \frac{d \ta}{\ta}\\
			 & \le \int_0^1 \frac{2\sup_{x\in E,  |\xi|\le \ta,} \|f^{(n+1)}(x+\xi) - f^{(n+1)}(x)\|\ta}{\ta} \frac{d \ta}{\ta}\\
			 & \le \int_0^1 2H^{\be,\infty}(f^{(n+1)}) \ta^\be \frac{d \ta}{\ta} = \frac{2}{\be} H^{\be,\infty}(f^{(n+1)}),\\
			 & \le \frac{2}{\be (1-2^{\be-1})} Z^{\be,\infty}(f^{(n+1)}) \le  \frac{2}{\be (1-2^{\be-1})} \|f\|^\cZ_{n+1,\be,\infty},
		\end{align*}
		where we used Lemma \ref{normequiv} to get the second to last inequality. Since  $\int_1^\infty \frac{\et(f^{(n)};\ta)}{\ta} \frac{d\ta}{\ta} \le 4\|f^{(n)}\|_\infty$, we get
		\[
		Z^{1,1}(f^{(n)}) \le \frac{6}{\be (1-2^{\be-1})} \|f\|^\cZ_{n+1,\be,\infty}.
		\] 
	\end{proof}

	\begin{remark}
		\label{rem:ex}
		But observe that $\cB^{n,1,p}_\cH$ (for $p < \infty$) does not quite fit in this picture. To this end, we construct a function $f \in \cB^{1,\al,\infty}_\cH \backslash \bigcup_{1 \le p < \infty}\cB^{0,1,p}_\cH$ . Let $f \in C^\infty(\R)$ with
		\[f(x) =
		\begin{cases}
		1&\text{ for } x>1,\\
		x&\text{ for } -1/2<x<1/2,\\
		-1&\text{ for } ~x< -1.
		\end{cases}
		\]
		Then clearly $f \in C^2_b(\R)$ and thus also in $\cB^{1,\al,\infty}_\cH$, but 
		\begin{align*}
			\int_0^1 \bigg(\frac{\om(f;\ta)}{\ta}\bigg)^p \frac{d\ta}{\ta} \ge \int_0^1 \frac{d\ta}{\ta} = \infty.
		\end{align*}
	\end{remark}
	
	\subsection{Summary}
	
	Let us visualize the inclusion relations in a diagram: Let $1\le p \le q \le \infty$, $0<\al < \be < 1$, $n \in \N$, then
	\[
	\xymatrix{
		\cB^{n+1,\al,q}  \ar@{^{(}->}[d] \\
		\cB^{n,1,p} \ar@{^{(}->}[r] & \cB^{n,1,q} \ar@{^{(}->}[r] & \cB^{n,1,\infty} \ar@{^{(}->}[r] & \cB^{n,\be,p} \ar@{^{(}->}[r] & \cB^{n,\be,q} \ar@{^{(}->}[r] & \cB^{n,\al,p}\\
		\cB^{n,1,p}_\cH \ar@{^{(}->}[u] \ar@{^{(}->}[r] & \cB^{n,1,q}_\cH \ar@{^{(}->}[u] \ar@{^{(}->}[r] & \cB^{n,1,\infty}_\cH \ar@{^{(}->}[u] \ar@{^{(}->}[r] & \cB^{n,\be,p}_\cH \ar@{=}[u] \ar@{^{(}->}[r] & \cB^{n,\be,q}_\cH \ar@{^{(}->}[r] \ar@{=}[u] & \cB^{n,\al,p}_\cH \ar@{=}[u]\\
		&&\cB^{n+1,\al,q} \ar@{^{(}->}[u]
		}
	\]
	where domain and codomain are some Banach spaces $E,F$. 
	The only inclusion not yet proved is $\cB^{n+1,\al,q} \hookrightarrow \cB^{n,1,\infty}_\cH$. This follows immediately from $H^{1,\infty}(f^{(n)})\le \|f^{(n+1)}\|_\infty$.
	
	\section{Composition in Besov spaces} Before we investigate the regularity of the composition operator, we include a crucial result needed in the proof of ODE-closedness.
	
	\begin{proposition}
		\label{multilinearbesov}
		Let $g_1 \in \cB^{0,\al,p}(\R^d,L_{k-1}(\R^d;\R^d))$ and $g_i \in \cB^{0,\al,p}(\R^d,L_{\ga_i}(\R^d;\R^d))$ for $2\le i \le k$. Then $x \mapsto \Ps(x) := g_1(x)\cdot (g_2(x), \dots , g_k(x)) \in \cB^{0,\al,p}(\R^d,L_n(\R^d;\R^d))$ where $n = \ga_2+\cdots+\ga_k$ and there exists a constant $C$ depending only on $\al$ and $k$ such that 
		\[
		\|\Ps\|^\cZ_{0,\al,p}\le C \prod_{i=1}^k\|g_i\|_{0,\al,p}^\cZ.
		\]
	\end{proposition}
	
	The proof of Proposition \ref{multilinearbesov} is based on the following easy observation.
	
	\begin{lemma}
		\label{multilinearexp}
		Let $T \in L_k(E_1, \dots , E_k;F)$ and $a=(a_i)_{i=1}^k,b=(b_i)_{i=1}^k,c=(c_i)_{i=1}^k \in \prod_{i=1}^{k}E_i$. Then $T(a)-2T(b)+T(c)$ can be written as a finite sum of elements of the form $T(\dots, a_i-2b_i+c_i,\dots)$, for some $1\le i \le k$, and $T(\dots, b_i-c_i,\dots,a_j-c_j,\dots)$, for some $1\le i < j \le k$. The dots represent entries of the form $a_l,~b_l$, or $c_l$.
	\end{lemma}
	
	\begin{proof}
		It is easily seen that for $f,g,h \in L(E;F)$, and $a,b,c \in E$ 
		\begin{equation}
		\label{indstart}
		f(a)-2g(b)+h(c)= (f-2g+h)(a) + g(a-2b+c)+(g-h)(a-c).
		\end{equation}
		Using \eqref{indstart}, we prove the claim by induction on $k$. For $k = 1$, the result is trivial due to linearity of $T$. So now let $k>1$. First we observe that for fixed $x_i \in E_i$, the mapping $x \mapsto T(x_1,\dots,x_{k-1},x)=:T(x_1,\dots,x_{k-1})(x)$ is in $L(E_k;F)$. So by applying \eqref{indstart}, we get that
		\begin{align*}
			&T(a)-2T(b)+T(c) \\
			&= (T(a_1,\dots,a_{k-1}) - 2T(b_1,\dots,b_{k-1})+T(c_1,\dots,c_{k-1}))(a_k)\\
			&~+T(b_1,\dots,b_{k-1},a_k-2b_k+c_k)\\
			&~+(T(b_1,\dots,b_{k-1})-T(c_1,\dots,c_{k-1}))(a_k-c_k).
		\end{align*}
		For the first summand, we apply the inductive assumption and get terms of the desired form, and for the third summand we expand $T(b_1,\dots,b_{k-1})-T(c_1,\dots,c_{k-1})$ by inserting a telescopic sum of terms $T(b_1,\dots,b_l,c_{l+1},\dots,c_{k-1})$, which yields the desired representation.
	\end{proof}
	
	\begin{proof}[Proof of Proposition \ref{multilinearbesov}]
		Let 
		\begin{align*}
			T:&~L_{k-1}(\R^d;\R^d)\times L_{\ga_2}(\R^d;\R^d)\times \cdots \times L_{\ga_k}(\R^d;\R^d) \rightarrow L_n(\R^d;\R^d)\\
			&~(B_1,B_2,\dots,B_k) \mapsto B_1\cdot(B_2, \cdots, B_k).
		\end{align*}
		Then $T$ is $k$-linear and $\Ps(x)=T(g_1(x), \dots, g_k(x))$, and we get by Lemma \ref{multilinearexp}
		\[
		\De^2_h\Ps(x)= \sum S(x,h),
		\]
		where $S(x,h)$ is of one of the following forms
		\begin{gather}
		\label{f1}
			T(\dots,\De^2_hg_i(x),\dots),\\
		\label{f2}
			T(\dots,g_i(y)-g_i(z), \dots, g_j(y)-g_j(z), \dots),
		\end{gather}
		here $y,z \in \{x,x+h,x+2h\}, y \neq z$, the dots are entries of the form $g_l(y)$, and $y,z$ may differ from entry to entry. Thus we can estimate
		\begin{align}
		\label{multilinearest}
			Z^{\al,p}(\Ps) \le \sum \bigg(\int_0^\infty \bigg(\frac{\sup_{x \in \R^d,\|h\|\le t}\|S(x,h)\|}{t^\al}\bigg)^p\frac{dt}{t}\bigg)^{1/p}.
		\end{align}
		For $S(x,h)=~$\eqref{f1} we estimate by 
		\begin{align*}
		Z^{\al,p}(g_i)\cdot \prod_{1\le l \le k, l\neq i} \|g_l\|_\infty \le \prod_{l=1}^k \|g_l\|^\cZ_{0,\al,p},
		\end{align*}
		and for $S(x,h)=~$\eqref{f2} by 
		\begin{align*}
			&~\bigg(\prod_{1\le l \le k, l\neq i,j}\|g_l\|_\infty \bigg)4\bigg(\int_0^\infty \bigg( \frac{\om(g_i;t) \om(g_j;t)}{t^\al}\bigg)^p \frac{dt}{t}\bigg)^{1/p}\\
			\le & ~4\bigg(\prod_{1\le l \le k, l\neq i,j}\|g_l\|_\infty \bigg) H^{\al/2,p}(g_i) H^{\al/2,\infty}(g_j)\\
			\le & ~4\bigg(\prod_{1\le l \le k, l\neq i,j}\|g_l\|_\infty \bigg)B \|g_i\|^\cZ_{0,\al,p} \|g_j\|^\cZ_{0,\al,p}\\
			\le & ~ 4 B\prod_{l=1}^k \|g_l\|^\cZ_{0,\al,p},
		\end{align*}
		 for a constant $B$ only depending on $\al$ due to Lemma \ref{normequiv} and Proposition \ref{embedding}. Since the number of summands in \eqref{multilinearest} only depends on $k$ we thus get a constant $C$ such that $Z^{\al,p}(\Ps) \le C \prod_{l=1}^k \|g_l\|^\cZ_{0,\al,p}$. Since we clearly also have
		 \[
		 \|\Ps\|_\infty \le \prod_{l=1}^k \|g_l\|_\infty,
		 \]
		the proof is finished.		
	\end{proof}
	
	The following is a variation of Lemma 2 from \cite{BourdaudCristoforis02} but for general (co)domains. The proofs are analogous.
	
	\begin{proposition}
		\label{comp}
		Let $0<\al,\be \le 1$, $p \in [1,\infty]$, $n \in \N_{\ge 1}$, $f \in \cB^{0,\al,p}(E,F)$, and $g\in \cB^{n,\be,p}(E,E)$. Then $f\circ g$ and $f \circ (\Id+g)$ are in $\cB^{0,\al,p}(E,F)$ and there exists a continuous increasing function $\ps: \R_+ \rightarrow \R_+$ independent of $f,g$ such that
		\begin{gather*}
		\|f \circ g\|^\cZ_{0,\al,p} \le \|f\|^\cZ_{0,\al,p}\ps(\|g\|^\cZ_{n,\be,p}),\\
		\|f \circ (\Id+g)\|^\cZ_{0,\al,p} \le \|f\|^\cZ_{0,\al,p}\ps(\|g\|^\cZ_{n,\be,p}).
		\end{gather*}
	\end{proposition}

	\section{Pointwise time-dependent Besov vector fields and their flows}

	\label{odeclosed} 
	
	Let $I = [0,1]$ and let $E$ be a Banach space of mappings from $\R^d$ to $\R^d$, which continuously embeds in $C^1_b(\R^d,\R^d)$.
	
	\begin{definition}
		We say that a mapping $u : I \times \R^d \to \R^d$ is a \emph{pointwise time-dependent $E$-vector field} if 
		the following conditions are satisfied:
		\begin{itemize}
			\item $u(t, \cdot) \in E$ for every $t \in I$.
			\item $u(\cdot,x)$ is measurable for every $x\in \R^d$.
			\item $I \ni t \to \|u(t,\cdot)\|_{E}$ is (Lebesgue) integrable. 
		\end{itemize}
		Let us denote the set of all pointwise time-dependent $E$-vector fields by $\fX_E(I,\R^d)$. 
		We remark that instead of the third condition we could also require that $t \mapsto \|u(t,\cdot)\|_E$ is dominated a.e.\ by some 
		non-negative function 
		$m \in L^1(I)$.
	
	We call $\Ph:I\times \R^d\rightarrow \R^d$ the \emph{flow} of $u \in \fX_E$ if it fulfills
	\begin{equation}
		\label{flow}
		\Ph(t,x) = x+\int_0^t u(s,\Ph(s,x))\,ds
	\end{equation}
	for all $t \in I,x \in \R^d$. To emphasize the dependence on $u$, we write $\Ph_u(t,x)$. In addition, we set $\ph_u(t,x):=\Ph_u(t,x)-x$.\par
	Finally, a Banach space $E$ is called \emph{ODE-closed} iff $\ph_u(t,\cdot) \in E$ for all $u \in \fX_E$ and all $t \in I$.
	\end{definition}
	
	Since $\cB^{n,\al,p}(\R^d,\R^d)$ continuously embeds in $C^1_b(\R^d,\R^d)$ for $n \ge 1$, we have in particular all notions of the previous definition for those Besov spaces.
	
	\begin{convention*}
		From now on, if we do not mention domain and codomain, we implicitly assume them to be $\R^d$.
		For pointwise time-dependent Besov vector fields, we write $\fX_{n,\al,p}$ instead of $\fX_{\cB^{n,\al,p}}$. For a function $f$ of two variables, such as $u$ or $\Ph_u$, we will also write $f(t)$ for $f(t,\cdot)$.
	\end{convention*}
	
	One of the main ingredients in showing ODE-closedness with pointwise techniques consists of Gronwall's inequality. We will only use the following weak version.
	
	\begin{lemma}
		\label{gronwall}
		Let $t_0 \in J\subseteq \R$ be an interval, $\al,~ u$ positive functions defined on $J$, and $c$ some positive constant. Assume that $u$ is bounded on $J$, and
		\[
		u(t) \le c + \int_{t_0}^t \al(s)u(s)\,ds.
		\]
		Then it follows
		\[
		u(t) \le c e^{\int_{t_0}^t \al(s)\,ds}.
		\]
	\end{lemma}
	
	Before we tackle ODE-closedness of Besov spaces, let us investigate the situation in $C^n_b$. Corresponding results for $C^n_0$, the subspace of functions vanishing together with their first $n$ derivatives at $\infty$, can be found in \cite[Thm. 8.7]{Younes10} for $n = 1$ and \cite[Thm. 8.9]{Younes10} for $n \ge 2$. A proof for $C^n_b$ can be given by making some minor adjustments. Hence we only indicate where the proofs have to be altered.

	\begin{proposition}
		\label{nodeclosed}
		For $n \ge 1$, $C^{n}_b$ is ODE-closed. For a fixed $u \in \fX_{C^n_b}$, the mapping $t \mapsto \ph_u(t)$ belongs to $C(I,C^n_b)$.
	\end{proposition}

\begin{proof}
	Let $n = 1$. Then in contrast to the proof of \cite[Thm. 8.7]{Younes10}, $Du(t)$ is in our case in general not uniformly continuous on all of $\R^d$, but certainly on bounded subsets. 
	Therefore for each $r > 0$
	\[
	\mu(t,\al,r):= \max\lbrace |Du(t,x)-Du(t,y)|: x,y \in B(0,r), ~|x-y|<\al \rbrace
	\]
	tends to $0$ as $\al \rightarrow 0$ for each fixed $t$. Then one substitutes the appearances of $\mu(t,\al)$ in \cite{Younes10} by $\mu(t,\al,r)$, where $r$ is sufficiently large. The rest of the proof remains the same.
	For $n \ge 2$ one makes similar changes. 
\end{proof}

\begin{remark}
	\label{extend}
	This shows that we can actually extend the results from \cite{NenningRainer16} concerning ODE-closedness of $C^{n,\be}_0$ and derive ODE-closedness of $\cC^{n,\be}$ ($C^{n,\be}_b$ in the notation of \cite{NenningRainer16}).
\end{remark}

Now we are in the position to prove our main theorem.
	
	\begin{namedthm*}{Theorem 1}
		\label{zygodeclosed}
		For all $n\in \N_{\ge 1},~\al \in (0,1],~ p \in [1,\infty]$, the Besov space $\cB^{n,\al,p}(\R^d,\R^d)$ is ODE-closed. For a fixed $u \in \fX_{n,\al,p}$, the mapping $t \mapsto \ph_u(t)$ belongs to $C(I,\cB^{n,\al,p}(\R^d,\R^d))$.	
	\end{namedthm*}

	\begin{proof}
	Fix $u \in \fX_{n,\al,p}$, and let $\Ph_u=\Id+\ph_u$ be the corresponding flow, i.e.
	it fulfils \eqref{flow}. Let us first collect what we already know:
		\begin{itemize}
			\item[(i)] For $\al \in (0,1)$ and $p = \infty$, using Remark \ref{extend}, the result is already proved in \cite{NenningRainer16}, since $\cZ^{n,\al} \cong \cC^{n,\al}$. So we are left to show the Theorem for $(n,\al,p) \in \N_{\ge 1}\times (0,1)\times [1,\infty) \cup \N_{\ge 1} \times \{1\}\times [1,\infty]$.
			\item[(ii)] From Proposition \ref{embedding}, we get $\fX_{n,\al,p} \subseteq \fX_{n,\frac{3\al}{4}, \infty}$ which yields due to \cite{NenningRainer16} that 
			\[
			t \mapsto \ph_u(t) \in C(I, \cB^{n,\frac{3\al}{4}, \infty})\subseteq C(I, \cB^{n,\al/2,p}) \subseteq C(I, \cB^{n,\al/2,\infty}).
			\]
			This implies that $\|\ph_u(t)\|^\cZ_{n,\frac{3\al}{4},\infty}$ and, since $\frac{3\al}{4} < 1$, $\|\ph_u(t)\|^\cH_{n,\frac{3\al}{4},\infty}$ are bounded uniformly for all $t \in I$.\par 
			In particular we find a constant $C$ such that for $t \in I$
			\begin{itemize}
				\item $\|d^k_x\Ph_u(t)\|_\infty \le C$ for $1\le k \le n$,
				\item $\|d^k_x \Ph_u(t)\|_{0,\al,p} \le C$ for $1\le k \le n-1$,
				\item $H^{\al/2,\infty}(d^k_x \Ph_u(t)),~H^{\al/2,p}(d^k_x \Ph_u(t)) \le C $ for $1\le k \le n$.
			\end{itemize}
			\item[(iii)] We may assume in addition that $C$ is chosen such that for all $k \le n$
			\[
			\|u(t)^{(k)}\circ \Ph_u(t)\|_{0,\al,p} \le \|u(t)^{(k)}\|_{0,\al,p}\ps (\|\ph_u(t)\|_{n,\al/2,p})\le C \|u(t)\|_{n,\al,p},
			\]
			where we used (ii), Proposition \ref{comp}, and the fact that $u(t)^{(k)} \in \cB^{0,\al,p}$. This implies
			\[
			H^{\al/2,p}(u(t)^{(k)}\circ \Ph_u(t)) \le C \|u(t)\|_{n,\al,p}.
			\]
		\end{itemize}

		By taking the $n$-th derivative with respect to the spatial variable in \eqref{flow}, we get that $\De^2_h d^n_x\Ph_u(t,x)$ satisfies the following ODE
		
		\begin{align*}
		\De^2_h d^n_x \Ph_u(t,x) = \int_0^t \De^2_h d^n_x (u(s)\circ \Ph_u(s,x))\,ds,
		\end{align*}
		and by applying Fa\`a di Bruno's formula to the integrand, this equals
		\begin{align*}
		\int_0^t& u(s)'(\Ph_u(s,x+h))\cdot\De^2_hd^n_x\Ph_u(s,x) + (\De^2_h u(s)'\circ \Ph_u(s,x))\cdot d^n_x\Ph_u(s,x+2h)\\
		&\quad + (u(s)'(\Ph_u(s,x+h))-u(s)'(\Ph_u(s,x)))\cdot(d^n_x\Ph_u(s,x+2h) - d^n_x\Ph_u(s,x)) \\
		&\quad +\underbrace{\on{sym}_n \sum_{k=2}^n\sum_{\ga \in \Ga(k,n)} c_\ga\De^2_h u(s)^{(k)}(\Ph_u(s,x))\cdot(\Ph_u^{(\ga_1)}(s,x),\cdots,\Ph_u^{(\ga_k)}(s,x))}_{=:R(s,x,h)}\,ds,
		\end{align*}
		where we used that $\on{sym}_n$ commutes with the second difference operator. This immediately yields
		\begin{align}
		\begin{split}
		\label{refineq}
		\et(d^n_x \Ph_u(t)&;\ta) \le \int_0^t \|u(s)'\|_\infty\et(d^n_x \Ph_u(s);\ta)+C\et(u(s)'\circ \Ph_u(s);\ta)\\
		&+ 2\om(u(s)'\circ\Ph_u(s);\ta)\om(d^n_x\Ph_u(s);\ta) + \sup_{x \in \R^d, \|h\|\le \ta}\|R(s,x,h)\|\,ds,
		\end{split}
		\end{align}
		where $\sup_{x \in \R^d,s\in I}\|d^n_x\Ph_u(s,x)\| \le C$ due to (ii). We claim that this yields 
		\begin{equation}
		\label{gronwn}
		Z^{\al,p}(d^n_x\Ph_u(t),\de) \le \int_0^t \|u(s)\|^\cZ_{n,\al,p}Z^{\al,p}(d^n_x\Ph_u(s),\de) + B \|u(s)\|^\cZ_{n,\al,p}\,ds,
		\end{equation}
		where $Z^{\al,p}(d^n_x\Ph_u(t),\de):= \big( \int_{\de}^{\infty}\big( \frac{\et(d^n_x\Ph_u(t);\ta)}{\ta^\al} \big)^p \frac{d\ta}{\ta} \big)^{1/p}$, and
		for some constant $B$ independent from $\de$. Since $Z^{\al,p}(d^n_x\Ph_u(t),\de)$ is bounded in $t$ for each fixed $\de$, we can apply Lemma \ref{gronwall} to \eqref{gronwn} and get
		\[
		Z^{\al,p}(d^n_x\Ph_u(t),\de) \le B \bigg(\int_0^t\|u(s)\|^\cZ_{n,\al,p}\,ds\bigg)e^{\int_0^1\|u(s)\|^\cZ_{n,\al,p}\,ds},
		\]
		which is independent from $\de$ and therefore also holds in the limit as $\de$ tends to $0$. This gives $\ph_u(t) \in \cB^{n,\al,p}$ for all $t$, i.e. ODE-closedness of $\cB^{n,\al,p}$; continuity in time will be argued afterwards. So we are left to show \eqref{gronwn}:\par 
		First assume $p < \infty$, we take the $p$-th power of \eqref{refineq}, divide by $\ta^{\al p + 1}$, integrate with respect to $\ta$ from $\de$ to $\infty$, take this term to the power $1/p$ and apply Minkowski's integral inequality. This procedure yields after applying the triangle inequality for the $p$-norm the following estimate
		\begin{align}
			\begin{split}
			\label{intest}
			Z^{\al,p}(d^n_x\Ph_u(t),\de) &\le \int_0^t \|u(s)\|^\cZ_{n,\al,p}Z^{\al,p}(d^n_x\Ph_u(s),\de) + CZ^{\al,p}(u(s)'\circ \Ph_u(s))\\
			&\quad + 2\bigg(\int_0^\infty \bigg(\frac{\om(u(s)'\circ\Ph_u(s);\ta)}{\ta^{\al/2}}\bigg)^p \bigg(\frac{\om(d^n_x\Ph_u(s);\ta)}{\ta^{\al/2}}\bigg)^p \frac{d\ta}{\ta} \bigg)^{1/p}\\
			&\quad + \bigg(\int_0^\infty \bigg(\frac{\sup_{x \in \R^d, \|h\|\le \ta}\|R(s,x,h)\|}{\ta^\al}\bigg)^p \frac{d\ta}{\ta}\bigg)^{1/p} \,ds.
			\end{split}
		\end{align}
		The second summand of the integrand can be estimated by $B\|u(s)\|^\cZ_{n,\al,p}$ due to (iii), where from now on $B$ is a generic constant. For the third summand, we first estimate 
		$\frac{\om(d^n_x\Ph_u(s);\ta)}{\ta^{\al/2}}$ by $B$ uniformly in $\ta$ and $s$ due to (ii), and for $H^{\al/2,p}(u(s)'\circ\Ph_u(s))\le B\|u(s)\|^\cZ_{n,\al,p}$, we apply (iii).\par 
		So we are left to investigate 
		\[
		\bigg(\int_0^\infty \bigg(\frac{\sup_{x \in \R^d, \|h\|\le \ta}\|R(s,x,h)\|}{\ta^\al}\bigg)^p \frac{d\ta}{\ta}\bigg)^{1/p},
		\]
		where we first use that $\|\on{sym} T\| \le \|T\|$ for any multilinear mapping $T$ and then apply the triangle inequality for the $p$-norm. We are thus left to find estimates of the form
		\begin{align}
		\begin{split}
		\label{intform}
		Z^{\al,p}((u(s)^{(k)}\circ \Ph_u(s))\cdot(\Ph_u^{(\ga_1)}(s), \dots, \Ph_u^{(\ga_k)}(s)))
		\le B\|u(s)\|^\cZ_{n,\al,p},
		\end{split}
		\end{align}
		for $k \ge 2$ which implies $\ga_i \le n-1$.
		By Proposition \ref{multilinearbesov}, (ii) and (iii), we get an estimate 
		\begin{align*}
		Z^{\al,p}((u(s)^{(k)}\circ& \Ph_u(s))\cdot(\Ph_u^{(\ga_1)}(s), \dots, \Ph_u^{(\ga_k)}(s))) \le C\|u(s)^{(k)}\|_{0,\al,p}\prod_{i = 1}^{k} \|\Ph_u^{(\ga_i)}(s)\|_{0,\al,p}\\
		&\le B\|u(s)\|^\cZ_{n,\al,p}.
		\end{align*}
		This yields \eqref{gronwn} and, as mentioned before, ODE-closedness of $\cB^{n,\al,p}$.
		Finally, to get continuity in time one observes that $\et(d^n_x \Ph_u(t)-d^n_x \Ph_u(r);\ta)$ can be estimated by the integral from $r$ to $t$ of the same integrand as in \eqref{refineq}. This gives by the same reasoning as above that 
		\begin{equation}
		\label{conteq}
		Z^{\al,p}(d^n_x\Ph_u(t)-d^n_x\Ph_u(r),\de) \le \int_r^t \|u(s)\|^\cZ_{n,\al,p}Z^{\al,p}(d^n_x\Ph_u(s),\de) + B \|u(s)\|^\cZ_{n,\al,p}\,ds,
		\end{equation}
		and Lemma \ref{gronwall} gives 
		\[
		Z^{\al,p}(d^n_x\Ph_u(t)-d^n_x\Ph_u(r),\de) \le B\bigg(\int_r^t\|u(s)\|^\cZ_{n,\al,p}\,ds\bigg)e^{\int_0^1\|u(s)\|^\cZ_{n,\al,p}\,ds},
		\]
		which again is independent from $\de$. Since $\int_r^t\|u(s)\|^\cZ_{n,\al,p}\,ds$ tends to $0$ as $r \rightarrow t$, we get continuity in time. The arguments used above to derive \eqref{gronwn} and \eqref{conteq} worked for $p < \infty$. For $p = \infty$, the situation actually gets simpler. One just has to substitute Minkowski's integral inequality with exchanging $\sup_{\ta > \de}$ with integration. The rest works analogously.
	\end{proof}

	\section{Continuity of the flow mapping}
	
	We investigate regularity properties of the map sending a vector field $u$ to $\on{Fl}(u):=\Ph_u-\Id = \ph_u$, the so-called flow mapping, for Besov vector fields. In order to speak of continuity, we have to fix a topology on domain and codomain. Let us first consider the domain, i.e. the set of vector fields of a certain Besov regularity. In \cite{NenningRainer16}, we were forced to use the smaller space of Bochner integrable (H\"older) vector fields. We will do the same here. So 
	\[
	\on{dom}(\on{Fl}) = L^1(I,\cB^{n,\be,p}),
	\]
	which is a normed space with $\|u\|:= \int_0^1\|u(t)\|_{n,\be,p}\,dt$.
	For a bare minimum of properties of this integrability notion, consult e.g. section 2.2 in \cite{NenningRainer16}.
	
	Now let us consider the codomain. As it is already stated in \cite{NenningRainer16}, we do not know whether we have continuity in the H\"older case without loss of regularity, i.e. of the mapping $\on{Fl}:L^1(I,\cB^{n,\be,\infty}_\cH) \rightarrow C(I,\cB^{n,\be,\infty}_\cH)$. The same problem arises for $p < \infty$. But by using Proposition \ref{embedding} together with the continuity results in \cite{NenningRainer16}, we can almost immediately deduce our second result.
	\begin{namedthm*}{Theorem 2}
		Let $0<\al < \be \le 1$ and $n \in \N_{\ge 2}$. Then 
		\[
		\on{Fl}: L^1(I,\cB^{n,\be,p}) \rightarrow C(I,\cB^{n,\al,p}), ~ u \mapsto \ph_u
		\]
		is continuous, even H\"older continuous of any order $\ga < \be-\al$.
	\end{namedthm*}
	
	\begin{proof}
		Just observe for $0<\varepsilon<(\be-\al)/2$
		\[
		L^1(I,\cB^{n,\be,p}) \hookrightarrow L^1(I,\cB^{n,\be-\varepsilon,\infty}_\cH) \overset{\on{Fl}}{\rightarrow} C(I,\cB^{n,\al+\varepsilon,\infty}_\cH) \hookrightarrow  C(I,\cB^{n,\al,p}),
		\]
		and $\on{Fl}$ is by Theorem 5.6. in \cite{NenningRainer16} $(\be-\al-2\varepsilon)$-H\"older continuous. Since $\varepsilon$ is arbitrary we get H\"older continuity of any order $\ga < \be-\al$.
	\end{proof}

\end{document}